\documentclass[12pt,a4paper, reqno]{amsart}
\usepackage{amsfonts,amsmath,amsthm,amssymb, amscd}
\usepackage[top=2cm, bottom=5cm, left=3cm, right=1cm]{geometry}

\def\gp#1{\langle#1\rangle}

\newtheorem{theorem}{Theorem}
\newtheorem{lemma}{Lemma}
\newtheorem{corollary}{Corollary}

\newtheorem{example}{Example}
\newtheorem{conjecture}{Conjecture}
\newtheorem{remark}{Remark}

\title[Maximal commutative subalgebras]{Maximal commutative subalgebras of a Grassmann algebra}

\author{Victor A.~Bovdi, Ho-Hon Leung}
\address{UAEU, Al-Ain, United Arab Emirates}
\email{ vbovdi@gmail.com,  hohon.leung@uaeu.ac.ae}

\keywords {Grassmann algebra, exterior algebra, maximal commutative subalgebra}
\thanks{Supported by the UAEU grants:  UPAR  Grant G00002599 and StartUp Grant: G00002235}
\subjclass{15A75, 05D05,  16W55, 13A02}

\begin{document}
\maketitle

\begin{abstract}
We investigate the structure of maximal commutative subalgebras of the finite dimensional  Grassmann algebra  over  a field  of characteristic different from two.
\end{abstract}
\section{Introduction}

The Grassmann algebra  (exterior algebra)  $G(n)$ over a field $F$ of characteristic different from two  is the following  finite dimensional associative algebra of rank $n$:
\begin{align}\label{gwbRG}
G(n)&=F[x_1,\ldots, x_n]/\gp{x_ix_j+x_jx_i\mid 1\leq i,j\leq n}_F.
\end{align}
The Grassmann algebra is widely used in ring theory,  differential geometry and the theory of manifolds
(for example, see \cite{Bokut_Kharchenko, Borisenko_Nikolaevskii, Kemer, Schlote}).

It is well known (for example, see \cite{Bokut_Kharchenko}) that $\mathrm{dim}_{F} G(n)=2^n$  and the identity $[[x,y],z]=0$ is satisfied for all $x,y,z\in G(n)$.  The algebra $G(n)$ has a large center and it is natural to investigate the structure of commutative subalgebras in $G(n)$. It was recently studied by Domokos and Zubor \cite{Domokos_Zubor}. Historically, the structure of commutative subalgebras are widely studied in certain rings and algebras (for example, see \cite{Bavula, Courter, Guterman_Markova, Hartwig_Oinert, Panyushev_Yakimova, Song, Zelazko}).

When $n$ is even, the structure of maximal commutative subalgebras (with respect to inclusion) in $G(n)$ is quite well-understood. In particular, Domokos and Zubor showed that all such maximal commutative subalgebras in $G(n)$ of even rank $n$ have dimension $3\cdot 2^{n-2}$ (Corollary 2.4, Theorem 7.1 (i) in \cite{Domokos_Zubor}) despite the fact that not all of them are isomorphic (Theorem 7.1 (ii) in \cite{Domokos_Zubor}).

However, the structure of maximal commutative subalgebras (with respect to inclusion) in $G(n)$ of odd rank $n$ is less clear. Domokos and Zubor obtained the maximum dimension of commutative subalgebras (Theorem 7.1 (i) in \cite{Domokos_Zubor}). Some partial results on the structure of maximal commutative subalgebras (with respect to inclusion) were obtained for $n=5$ and $n=7$ (Proposition 7.5 in \cite{Domokos_Zubor}).

The main goal of this paper is to provide new constructions of maximal commutative subalgebras (with respect to inclusion) of $G(n)$ of odd rank $n$. Our constructions are based on certain intersecting families of subsets of odd size in the power set of $n$. We make the connection between maximal commutative subalgebras of $G(n)$ and such combinatorial objects explicit in Section \ref{S:1}. In Section \ref{section2}, we introduce an intersecting family of subsets of odd size based on the classical Fano plane. It is the foundation of our constructions in Section \ref{Section3} and Section \ref{Section4}.

Based on our results in Section \ref{Section3} and Section \ref{Section4}, we make inroads into solving two  conjectures posted by Domokos and Zubor (Conjecture 7.3 and Conjecture 7.4 in \cite{Domokos_Zubor}). Along the way we state a few conjectures, some of which are purely computational in nature, related to our constructions of maximal commutative subalgebras in $G(n)$ of odd rank $n$.

\section{Combinatorics of the Grassmann algebra} \label{S:1}

Let $m,n\in \mathbb{N}$ such that $m \leq n$.    For two integers $i$ and $j$ ($i\leq j$), we write
\[
[i, j] = \{k \in  \mathbb{Z} \mid i \leq k \leq j\}.
\]
The set $[1, k]$ is  often abbreviated to  $[k]$. The  cardinality of a set $J$ is denoted by $|J|$.

In the Grassmann algebra   $G(n)$   of rank $n$ we can choose the basis $\{1,  v_1, \ldots, v_n\}$ such that
\[
G(n)=F\gp{1, v_1, \ldots, v_n\mid   v_iv_j=-v_jv_i, \quad 1\leq i,j\leq n}.
\]
For $I=\{i_1,\ldots ,i_k\}\subset [n]$, we denote $v_I$ by $v_I:= v_{i_1} v_{i_2}\cdots v_{i_k}$. For any $v_I,v_J\in G(n)$,  in which  $I,J\subset [n]$, we have the following:
\begin{align}
v_I v_J &= (-1)^{|I| |J|}v_J v_I. \label{E:1}
\end{align}
If $|I|$ is odd, then $v_I$ is an  \underline{{\it odd element}} in $G(n)$,  otherwise it is an \underline{{\it even element}}.

Clearly, $G(n)=G_0\oplus G_1$ is a $\mathbb{Z}_2$-graded algebra, where the \underline{{\it even part}} is
\[
G_0=\bigoplus_{k\in 2\mathbb{N}\cup \{0\}}  {\rm span}_F\{v_J \mid  J \subseteq  [n]\; \text{and}\;  |J| = k\}
 \]
and the \underline{{\it odd part}} is
 \[
G_1=\bigoplus_{k\in  \mathbb{N}\setminus 2\mathbb{N}}  {\rm span}_F\{v_J \mid  J \subseteq  [n]\; \text{and}\;   |J| = k\}.
 \]
The collection of all subsets of $[n]$ (including the empty set $\emptyset$) and the collection of all subsets of size $m$ in the set $[n]$ are  denoted  by $\mathcal{P}_n$ and $\mathcal{P}_n(m)$, respectively.  Set
\[
\mathcal{P}_n(*)=\cup_{m=1}^n\mathcal{P}_n(m)=\mathcal{P}_n\setminus \emptyset.
 \]
The nonempty collection $\mathrm{\Delta}$ of subsets of odd size in $\mathcal{P}_n(*)$  is called an {\underline{\it algebraic system}} if,  for any  $A_1\in \mathcal{P}_n \setminus \mathrm{\Delta}$ such that $|A_1|$ is even and
for any $A_2\in\mathrm{\Delta}$ such that $A_1 \cap A_2 =\emptyset$,  we have
\[
A_1 \cup A_2\in \mathrm{\Delta}.
\]
An  algebraic system $\Delta$ is called {\underline{\it commutative}}  if  any two elements from $\Delta$  have a non-trivial intersection. A commutative system $\Delta$ is called {\underline{\it maximal}}  if, for any $A\in\mathcal{P}_n(*)\setminus  \Delta$ such that  $|A|$ is odd, there exists $B\in\Delta$ such that $A\cap B=\emptyset$.

\begin{example}  \label{example1}
For $n=4$, let $\Delta=\{\{1\},\{1,2,3\},\{1,3,4\},\{1,2,4\}\}$ and $\mathcal{D}=\{\{2\},\{2,3,4\},\newline  \{1,2,3\}\}$. Then $\Delta$ is algebraic and  $\mathcal{D}$ is not algebraic because $\{1,2,4\}=\{2\}\cup \{1,4\}\notin \mathcal{D}$.
\end{example}

A subalgebra in $G(n)$ is called \underline{{\it homogeneous}} if it is spanned by monomials. It is well known that a maximal commutative subalgebra of $G(n)$ is a homogeneous subalgebra  (see also Theorem 4.3 in \cite{Domokos_Zubor}, which says that  each  commutative subalgebra in $G(n)$ can be transformed into a homogeneous commutative one by a simple process of linear transformations).

Let $M$ be a maximal commutative subalgebra (with respect to inclusion) of the Grassmann algebra $G(n)=G_0\oplus G_1$, where $n\geq 2$. From now on, we focus solely on homogeneous commutative subalgebras of a Grassmann algebra ( see the comment in the previous paragraph). Clearly, $M$ can be written as $M=G_0\oplus M_1$ where $M_1\subset G_1$. Consider
\[
\nabla_M:=\{I\subseteq [n]\mid v_I \in M_1 \}.
\]
If  $T$ is  a maximal commutative algebraic system  in $[n]$, then we  define:
\[
M_{T}:={\rm span}_F\{ v_I \mid I\in T \}.
\]We have the following lemma:

\begin{lemma} \label{L:1}
If  $M$ is  a maximal (with respect to inclusion) commutative subalgebra of $G(n)$, then $\nabla_M$ is a maximal commutative algebraic system. Moreover, if $T$  is a maximal commutative algebraic system, then $G_0\oplus M_{T}$ is a maximal commutative subalgebra in $G(n)$.
\end{lemma}

\begin{proof}
Let $M=G_0\oplus M_1$. If $v_I, v_J\in M_1$, then $v_I$ and $v_J$ commute if and only if $I\cap J\neq \emptyset$. Hence $\nabla_M$ is commutative.

Let $S\in \mathcal{P}_n \setminus \nabla_M$, $|S|$ is odd and $S\cap S'\neq\emptyset$ for any $S'\in\nabla_M$. Then $v_S\notin M$, $v_S v_{S'} =0=v_{S'} v_S$ and hence $v_S$, $v_{S'}$ commute. So $v_S$ commutes with all elements in $M_1$ and  with all elements in $G_0$. So $M\cup \{v_S\}$ is a commutative subalgebra which contradicts the maximality of $M$. Consequently, $\nabla_M$ is maximal.

Let $v_{S''}\in M_0$. Then $v_{S''} M_1\subset M_1$ as $M$ is a subalgebra in $G(n)$. If $v_{S_0}\in M_1$, then $v_{S''}$ and $v_{S_0}$ commute by \eqref{E:1}. In particular, if $S'' \cap S_0 =\emptyset$, then $v_{S''} v_{S_0}=v_{S'' \cup S_0}\neq 0$. Hence $S''\cup S_0\in \nabla_M$,  so $\nabla_M$ is algebraic.

The second statement is obvious. \end{proof}

\begin{corollary}\label{C:1}
Finding maximal (with respect to inclusion) commutative subalgebras in $G(n)$ is equivalent to finding maximal commutative algebraic systems in $\mathcal{P}_n(*)$.
\end{corollary}
If $n$ is even, there are easy constructions of maximal commutative algebraic systems in $G(n)$. We state them as examples.

\begin{example} \label{example2}
Let $n=4k$ where $k>0$. Let $\mathcal{S}\subset\mathcal{P}_n(*)$ be the following collection of subsets: \[ \mathcal{S}=\cup_{i \text{ is }odd, 2k+1\leq i \leq 4k-1} \mathcal{P}_n(i).\]It is plain to check that $\mathcal{S}$ is a  maximal commutative algebraic system  in $\mathcal{P}_n(*)$. We note that \[
|\mathcal{S}|=\textstyle \binom{4k}{2k+1}+\binom{4k}{2k+3}+\dots+\binom{4k}{4k-1} =2^{4k-2}=2^{n-2}.
\]
\end{example}
By Lemma \ref{L:1}, we obtain a maximal commutative subalgebra $G_0\oplus M_{\mathcal{S}}$ in $G(4k)$, such that
\[
\dim(G_0\oplus M_{\mathcal{S}})=2^{n-1}+2^{n-2}=3\cdot 2^{n-2}.
\]

\begin{example} \label{example3}
Let $n=4k+2$ where $k>0$. Let $\mathcal{S}_1\subset\mathcal{P}_n(*)$ be the following collection of subsets:
\[
\mathcal{S}_1=\cup_{i \text{ is }odd, 2k+3\leq i \leq 4k+1} \mathcal{P}_n(i).
\]
For any $i\in [n]$, we define $\mathcal{S}_2(i)\subset\mathcal{P}_n(*)$ be the following collection of subsets:
\[
\mathcal{S}_2(i)= \{U\in\mathcal{P}_n(2k+1) \mid  i\in U \}.
\]
Set $\mathcal{S}(i):=\mathcal{S}_1\cup \mathcal{S}_2(i)$. The set  $\mathcal{S}$ is a  maximal commutative algebraic system  in $\mathcal{P}_n(*)$.
\end{example}
Indeed
\[
\textstyle
|\mathcal{S}(i)|=|\mathcal{S}_1|+|\mathcal{S}_2(i)|=\binom{4k+1}{2k}+\binom{4k+2}{2k+3}+\dots+\binom{4k+2}{4k+1}=2^{4k}=2^{n-2}.
\]
By Lemma \ref{L:1}, we obtain  a maximal commutative subalgebra $G_0\oplus M_{\mathcal{S}(i)}$ in $G(4k+2)$ such that
\[
\dim(G_0\oplus M_{\mathcal{S}(i)})=2^{n-1}+2^{n-2}=3\cdot 2^{n-2}.
\]
The algebra $G_0\oplus M_{\mathcal{S}(x)}$ is isomorphic to $G_0\oplus M_{\mathcal{S}(y)}$ for each  $x,y\in [n]$ (by a permutation of basis elements $v_i$ with  $i\in [n]$). Hence, there exist $4k+2(=n)$ pieces of pairwise isomorphic maximal commutative subalgebras of dimension $3\cdot 2^{n-2}$ in $G(4k+2)$ which are all isomorphic to $G_0\oplus M_{S(i)}$.

Let $S\subset [n]$ with $n\geq 2$. Define
\[
\vee^1 S:=\{S,  S\cup S_1\mid S_1\in \mathcal{P}_n(2),  S_1\cap S=\emptyset\}.
\]
For $i>1$, we define inductively $\vee^i S:= \vee^1 (\vee^{i-1}S)$.

We define a {\it cone} $\mathrm{Cone}(S)$ based  on $S$ as follows:
\[
\mathrm{Cone}(S):=(\vee^\infty S)\cap [n].
\]
If $\mathcal{B}\subseteq \mathcal{P}_n(*)$, then we define  $\vee^1 \mathcal{B}:=\cup\{\vee^1 S\text{  }|\text{  }S\in\mathcal{B}\}$. We also extend our definition of $\mathrm{Cone}(\mathcal{B})$ based on $\mathcal{B}$ naturally.

\begin{example}
For $n=4$, in Example \ref{example1}, we have $\Delta=\vee^1 \{1\}$. It is also the cone based on a singleton $\{1\}$, that is, $\Delta=\mathrm{Cone}(\{1\})$.
\end{example}

\begin{example}
For $n=5$, $\Delta=\mathrm{Cone}(\{2\})$ is the set
\begin{align*}
\Delta&=\vee^2\{2\}= \vee^1\{\{2\},\{1,2,3\},\{1,2,4\},\{1,2,5\},\{2,3,4\},\{2,3,5\},\{2,4,5\} \}\\
   &= \{\{2\},\{1,2,3\},\{1,2,4\},\{1,2,5\},\{2,3,4\},\{2,3,5\},\{2,4,5\},\{1,2,3,4,5\} \}.
\end{align*}
\end{example}

\begin{lemma} \label{L:2}
The following conditions hold:
\begin{itemize}
\item[(i)] if $S\in \mathcal{P}_n(*)$, then $\mathrm{Cone}(S)$ is a commutative algebraic system;

\item[(ii)] if  $\Delta \subset \mathcal{P}_n$ is  a  maximal commutative algebraic system, then  for each  $S\in \Delta$ we have  $\mathrm{Cone}(S)\subset\Delta$;

\item[(iii)] $\mathrm{Cone}(\{i\})$ is a  maximal commutative algebraic system,    such that  $|\mathrm{Cone}(\{i\})|=2^{n-2}$.

\end{itemize}
\end{lemma}

\begin{proof}
(i) and (ii) are obvious.

$\mathrm{Cone}(\{i\})$ is algebraic by part (i). It is commutative since $\{i\}$ is a common subset for any element in $\mathrm{Cone}(\{i\})$ .

Let $S\in \mathcal{P}_n(*) \setminus\mathrm{Cone}(\{i\})$ such that $|S|$ is odd. Then $S\cap\{i\}=\emptyset$. Let $S_0$ be the complement of $S$ in $[n]$. Since $\{i\}\subset S_0$, $S_0\in \mathrm{Cone}(\{i\})$ if $|S_0|$ is odd. If $|S_0|$ is even, then let $S_1=S_0\setminus \{j\}$ where $i\neq j$. $S_1\in \mathrm{Cone}(\{i\})$ and $S_1\cap S=\emptyset$. Hence $\mathrm{Cone}(\{i\})$ is maximal.

Decompose $\mathrm{Cone}(\{i\})$ into the union of $\mathcal{U}_i$,  where $\mathcal{U}_i\subset \mathcal{P}_n(i)$. The set $\mathcal{U}_i$ is non-empty if and only if $i$ is odd. Clearly $\textstyle|\mathcal{U}_i|=\binom{n-1}{i-1}$, so
\[
\textstyle
|\mathrm{Cone}(\{i\})| =1 +\binom{n-1}{2}+\binom{n-1}{4}+\cdots+\binom{n-1}{2i'},
\]
where $2i'$ is the largest even number less than or equal to $n-1$. In the polynomial expansion of $(1+x)^{n-1}$, put $x=1$ and we get the classical identity on binomial distribution:
\[
\textstyle
2^{n-1}=\binom{n-1}{0} + \binom{n-1}{1}+\cdots+\binom{n-1}{n-1}.
\]
If we put $x=-1$ into the expansion of $(1+x)^{n-1}$ and rearrange terms, we  get
\[
\textstyle
|\mathrm{Cone}(\{i\})|=\frac{1}{2}\Big(\binom{n-1}{0} + \binom{n-1}{1}+\cdots+\binom{n-1}{n-1}\Big)
\]
and hence the result.
\end{proof}

Based on Lemmas \ref{L:1} and  \ref{L:2}, we construct a maximal commutative subalgebra $G_0 \oplus M_{\mathrm{Cone}(\{i\})}$ in $G(n)$ based on each $i\in[n]$  in $\mathrm{Cone}(\{i\})$. Now it is easy to see that
\[
\dim(G_0 \oplus M_{\mathrm{Cone}(\{i\})})=2^{n-1}+2^{n-2}=3\cdot 2^{n-2}.
\]
 Hence, we state the following corollary.

\begin{corollary}\label{C:2}
For each  positive integer $n$, there always exists $n$ pieces of pairwise  isomorphic maximal commutative subalgebras of dimension $3\cdot 2^{n-2}$ in  $G(n)$ which are all isomorphic to $G_0 \oplus M_{\mathrm{Cone}(\{i\})}$.
\end{corollary}

\begin{remark}
Let $i\in [n]$ where $n=4k$ with  $k>0$.  It is clear that the maximal commutative subalgebra $G_0\oplus M_{\mathcal{S}}$ (see Example \ref{example2}) and $G_0\oplus M_{\mathrm{Cone}(\{i\})}$ (see Lemma \ref{L:2}(iii)) are not isomorphic in $G(n)$ despite the fact that both of them have the same dimension $3\cdot 2^{n-2}$.
\end{remark}

\begin{remark}
Let $i\in [n]$, where  $n=4k+2$ with $k>0$.  It is clear that the maximal commutative subalgebra $G_0\oplus M_{S(i)}$ (see  Example \ref{example3}) and $G_0\oplus M_{\mathrm{Cone}(\{i\})}$ (see Lemma \ref{L:2}(iii)) are not isomorphic in $G(n)$ despite the fact that both of them have the same dimension $3\cdot 2^{n-2}$.
\end{remark}

\section{An intersecting family based on Fano System}\label{section2}

\noindent
{\underline {\it The Fano System}} (or {\it $0$-Fano System}) is the following collection $\mathcal{A}_3$ of subsets of the set $[7]$:
$A_1=\{1,2,5\}$, $A_2=\{1,3,6\}$, $A_3=\{1,4,7\}$, $A_4=\{2,3,7\}$,   $A_5=\{3,4,5\}$, $A_6=\{5,6,7\}$ and  $A_7=\{2,4,6\}$.
\noindent
{\underline {\it The extended Fano System}} (or {\it $4k$-Fano System}) is the collection of the following subsets of $[4k+7]$:
\begin{equation}\label{E:Qlb}
\begin{split}
\textstyle
\big\{\;  \{4k+1,4k+2,4k+5\}, \{4k+1,4k+3,4k+6\},\\
\{4k+1,4k+4,4k+7\}, \{4k+2,4k+3,4k+7\}, \\
\{4k+3,4k+4,4k+5\}, \{4k+5,4k+6,4k+7\},\\
 \{4k+2,4k+4,4k+6\}\; \big\}, \qquad (k\geq 1).
\end{split}
\end{equation}
It is easy to see that if $k=0$ then  we have the $0$-Fano system. The properties of the Fano system can be found in \cite{Polster, Lint_Wilson}.

The \underline{{\it $s$-bicommutative system}} is a collection $T$ of subsets of $\mathcal{P}_n(s)$ such that:
\begin{itemize}
\item[(i)] $T$ is  commutative (that is, any two elements in $T$ have a non-trivial intersection);
\item[(ii)] $T$ is  maximal among subsets of size $s$, i.e. if $S_0\not\in T$ and $|S_0|=s$, then there exists $S_1\in T$, such that $S_0\cap S_1 = \emptyset$.
\end{itemize}

Note that in general an $s$-bicommutative system can not be an algebraic system, because all elements of an $s$-bicommutative system have the same size.

\begin{lemma} \label{L:3}
The Fano System $\mathcal{A}_3$ is a $3$-bicommutative system in $\mathcal{P}_7(3)$.
\end{lemma}

\begin{proof}
Clearly  $A_i \cap A_j\neq \emptyset$, where  $i,j\in[7]$, because
\[
\begin{split}
A_1&\cap A_2=A_1\cap A_3=A_2\cap A_3=\{1\},\quad A_1\cap A_4=A_1 \cap A_7=A_4\cap A_7=\{2\},\\
A_2&\cap A_4=A_2\cap A_5=A_4\cap A_5=\{3\},\quad A_3\cap A_5=A_3\cap A_7=A_5\cap A_7=\{4\},\\
A_1&\cap A_5=A_1\cap A_6=A_5 \cap A_6=\{5\},\quad A_2\cap A_6=A_2\cap A_7=A_6\cap A_7=\{6\},\\
A_3&\cap A_4=A_3 \cap A_6=A_4\cap A_6=\{7\}.
\end{split}
\]
The  system $\mathcal{A}_3$ is maximal, i.e.  for any $S\in \mathcal{P}_7(3) \setminus \mathcal{A}_3$,  there exists $A_i\in\mathcal{A}_3$ such that $A_i \cap S=\emptyset$.  Indeed, case-by-case we have:

\begin{align*}
\{1,2,3\}\cap A_6&=\{1,2,4\}\cap A_6=\{1,3,4\}\cap A_6=\{2,3,4\}\cap A_6=\emptyset,\\
\{2,3,5\}\cap A_3&=\{2,3,6\}\cap A_3=\{2,5,6\}\cap A_3=\{3,5,6\}\cap A_3=\emptyset,\\
\{3,4,6\}\cap A_1&=\{3,4,7\}\cap A_1=\{3,6,7\}\cap A_1=\{4,6,7\}\cap A_1=\emptyset,\\
\{1,2,6\}\cap A_5&=\{1,2,7\}\cap A_5=\{1,6,7\}\cap A_5=\{2,6,7\}\cap A_5=\emptyset,\\
\{1,4,5\}\cap A_4&=\{1,4,6\}\cap A_4=\{1,5,6\}\cap A_4=\{4,5,6\}\cap A_4=\emptyset,\\
\{2,4,5\}\cap A_2&=\{2,4,7\}\cap A_2=\{2,5,7\}\cap A_2=\{4,5,7\}\cap A_2=\emptyset,\\
\{1,3,5\}\cap A_7&=\{1,3,7\}\cap A_7=\{1,5,7\}\cap A_7=\{3,5,7\}\cap A_7=\emptyset.
\end{align*}
\end{proof}

Now, let $n=4k+7$ and $m=2k+3$,  where $k\geq 2$. Put $\mathcal{S}'_7=[4k+1,4k+7]$.

Clearly $[n]=[4k+7]=\mathcal{S}'_7\cup{[4k]}$. Define $\mathcal{B}_i \subset \mathcal{P}_n(m)$  $(i\in[10])$ as follows:
\begin{align*}
\mathcal{B}_1 &= \{S\text{    }|\text{    }S=\{4k+1,4k+2,4k+5\}\cup S' \text{ such that } S'\subset{[4k]}\},\\
\mathcal{B}_2 &= \{S\text{    }|\text{    }S=\{4k+1,4k+3,4k+6\}\cup S' \text{ such that } S'\subset{[4k]}\},\\
\mathcal{B}_3 &= \{S\text{    }|\text{    }S=\{4k+1,4k+4,4k+7\}\cup S' \text{ such that } S'\subset{[4k]}\},\\
\mathcal{B}_4 &= \{S\text{    }|\text{    }S=\{4k+2,4k+3,4k+7\}\cup S' \text{ such that } S'\subset{[4k]}\},\\
\mathcal{B}_5 &= \{S\text{    }|\text{    }S=\{4k+3,4k+ 4,4k+5\}\cup S' \text{ such that } S'\subset{[4k]}\},\\
\mathcal{B}_6 &= \{S\text{    }|\text{    }S=\{4k+5,4k+6,4k+7\}\cup S' \text{ such that } S'\subset{[4k]}\},\\
\mathcal{B}_7 &= \{S\text{    }|\text{    }S=\{4k+2,4k+4,4k+6\}\cup S' \text{ such that } S'\subset{[4k]}\},\\
\mathcal{B}_8 &= \{S\text{    }|\text{    }S=S' \cup S'', |S'|=2, |S''|=2k+1, S'\subset\mathcal{S}'_7, S''\subset{[4k]}\},\\
\mathcal{B}_9 &= \{S\text{    }|\text{    }S=S' \cup S'', |S'|=1, |S''|=2k+2, S'\subset\mathcal{S}'_7, S''\subset{[4k]}\},\\
\mathcal{B}_{10} &= \{S\text{    }|\text{    }S\subset{[4k]}\}.
\end{align*}

We note the following intersecting property for elements in $\mathcal{B}_1,\ldots ,\mathcal{B}_7$.

\begin{lemma}\label{L:4}
If $S_1\in\mathcal{B}_i$ and $S_2\in\mathcal{B}_j$ $(i,j\in[7], i\not=j)$, then $S_1\cap S_2\neq \emptyset$.
\end{lemma}

\begin{proof}
We randomly pick exactly one element in each $\mathcal{B}_i$ and label them as $S_i$ respectively. Note that $\{S_1\cap\mathcal{S}'_7,\ldots ,S_7\cap\mathcal{S}'_7\}$ coincides with \eqref{E:Qlb} and it  is bijective to the Fano System $\mathcal{A}_3$ and hence Lemma \ref{L:3} applies to $\{S_1\cap\mathcal{S}'_7,\ldots ,S_7\cap\mathcal{S}'_7\}$ as a collection of subsets in $\mathcal{S}'_7$. Consequently  $S_i$ and $S_j$, $i\neq j$, must intersect non-trivially in $\mathcal{S}'_7\subset[n]$.
\end{proof}

\begin{theorem} \label{T:1}
Let $n=4k+7$ and $m=2k+3$, where  $k\geq 2$. Then
\[
\mathcal{A}_m:=\cup_{i=1}^{10} \mathcal{B}_i
\]
is an  $m$-bicommutative system in $\mathcal{P}_n(m)$.
\end{theorem}

\begin{proof}
We  show that any two elements in $\mathcal{A}_m$ intersect non-trivially.

It is clear that $S_1\cap S_2\neq \emptyset$ in $\mathcal{S}'_7 \subset [n]$ for  $S_1, S_2\in\mathcal{B}_j$ with  $j\in [7]$. If $S_1\in\mathcal{B}_i$,  $S_2\in\mathcal{B}_j$  and $i\neq j$ (with  $i,j\in [7]$), then $S_1$ and $S_2$ have a non-trivial intersection by Lemma \ref{L:4}.

If $S_1\in\mathcal{B}_i$ and $S_2\in\mathcal{B}_j$ with  $i,j\in[8,10]$, then
\[
|S_1\cap {[4k]}|+|S_2\cap{[4k]}|>4k.
\]
Sets $S_1$ and $S_2$ have a non-empty intersection in ${[4k]}\subset[n]$ by Pigeonhole Principle.

Similarly, if $S_1\in\mathcal{B}_i$ and $S_2\in\mathcal{B}_j$ where $i\in[7]$ and $j\in[8,10]$, then
\[
|S_1\cap{[4k]}|+|S_2\cap{[4k]}|=2k+|S_2\cap{[4k]}|>4k.
\]
The sets $S_1$ and $S_2$ have a  non-empty intersection in ${[4k]}\subset[n]$ by Pigeonhole Principle.

Let us prove that the  commutative system $\mathcal{A}_m$ is maximal.
If $S\in\mathcal{P}_n(m) \setminus \mathcal{A}_m$, then we can write $S=S'\cup S''$ where $S'\subset \mathcal{S}'_7$ and $S''\subset{[4k]}$.

Note that if $|S'|=2$, then $S\in\mathcal{B}_8\subset\mathcal{A}_m$ and it is a contradiction. Similarly, if $|S'|=1$, then $S\in\mathcal{B}_9\subset\mathcal{A}_m$ and it is a contradiction. If $|S'|=0$, then $S=S''\subset{[4k]}$ and hence $S\in\mathcal{B}_{10}\subset\mathcal{A}_m$, which is a contradiction. We need to consider $S$ such that $|S'|>2$. We need to show that there exists an element in $\mathcal{A}_m$ such that it has no non-trivial intersection with $S$.

If $|S'|=7$, then $S=\mathcal{S}'_7$, $|S''|=2k-4$ and hence $|{[4k]} \setminus S''|=2k+4$. We can choose any $2k+3$ elements in $({[4k]} \setminus S'')\subset{[4k]}$ and call this set of elements $S_1$. Note that $S_1\in\mathcal{B}_{10}\subset\mathcal{A}_m$ and $S_1\cap S=\emptyset$.

If $|S'|=6$, then $|S''|=2k-3$ and $|{[4k]}\setminus S''|=2k+3$. Let $S_2 ={[4k]} \setminus S''$. We note that $S_2\in\mathcal{B}_{10}\subset\mathcal{A}_m$, $S_2\cap S''=\emptyset$ and hence
\[
S_2 \cap  S=S_2 \cap ( S'\cup S'' )=(S_2 \cap S') \cup (S_2 \cap S'')=\emptyset.
\]
If $|S'|=5$, then $|S''|=2k-2$, $|\mathcal{S}'_7  \setminus S'|=2$ and $|{[4k]}  \setminus  S''|=2k+2$. We randomly choose one element in $\mathcal{S}'_7  \setminus S'$, call it $a$. Let $S_3:=\{a\}\cup({[4k]}  \setminus S'')$. Then $S_3\in\mathcal{B}_9\subset\mathcal{A}_m$ and $S_3\cap S=\emptyset$.

If $|S'|=4$, then
\[
|S''|=2k-1,\quad  |\mathcal{S}'_7  \setminus S'|=3\quad\text{ and}\quad  |{[4k]} \setminus  S''|=2k+1.
\]
We randomly choose two elements in $\mathcal{S}'_7  \setminus  S'$, say $b$ and $c$ respectively. Set  $S_4:=\{b,c\}\cup ({[4k]} \setminus S'')$. Evidently  $S_4\in\mathcal{B}_8\subset\mathcal{A}_m$ and $S_4\cap S=\emptyset$.

If $|S'|=3$, then $|S''|=2k$ and $|{[4k]}\setminus S''|=2k$. As $S\not\in \mathcal{A}_m$, the $k$-Fano System  does not contain $S'$. By Lemma \ref{L:3}, there exists an element in the $k$-Fano System in $\mathcal{S}'_7$ that has an empty intersection with $S'$, denote  this element by $F$. Set $S_5:=F\cup({[4k]} \setminus S'')$. Clearly  $S_5\cap S=\emptyset$ and $S_5\in\mathcal{B}_j\subset\mathcal{A}_m$ for some $j\in [7]$.
\end{proof}

%\section{Applications} \label{section3}

For $n=4k+7$ and $m=2k+3$ where $k\geq 2$, we define
\begin{align*}
\mathrm{Spec}(k,2)&:=\min \{ \text{ }|\mathcal{M}_{4k+7}| \text{ } \mid \text{ }  \mathcal{M}_{4k+7} \text{ is a } (2k+3)\text{-bicommutative system in }\mathcal{P}_{n}(m)\}.\end{align*}

By the classical Erd\H os-Ko-Rado Theorem \cite{Erdos_Ko_Rado}, we state the following lemma:

\begin{lemma} \label{L:5}
If $n=4k+7$ and  $m=2k+3$ with $k\geq 2$, then
\[
\textstyle
\mathrm{Spec}(k,2)\leq\binom{n-1}{m-1}=\binom{4k+6}{2k+2}.
\]
\end{lemma}

Based on our construction of the $m$-bicommutative system $\mathcal{A}_m$,
\begin{align*}
|\mathcal{A}_m| &= |\cup_{i=1}^7  \mathcal{B}_i|+|\mathcal{B}_8|+|\mathcal{B}_9|+|\mathcal{B}_{10}| \\
     &=
     7\textstyle \binom{4k}{2k}+\binom{7}{2}\binom{4k}{2k+1} +\binom{7}{1}\binom{4k}{2k+2} +\binom{4k}{2k+3}.
\end{align*}

If $k\rightarrow \infty$,
\begin{align*}
\frac{\binom{4k}{2k}}{\binom{4k+6}{2k+2}} &= \frac{(2k+1)(2k+2)(2k+1)(2k+2)(2k+3)(2k+4)}{(4k+1)(4k+2)(4k+3)(4k+4)(4k+5)(4k+6)}\rightarrow\frac{1}{64}, \\
\frac{\binom{4k}{2k+1}}{\binom{4k+6}{2k+2}} &= \frac{(2k)(2k+1)(2k+2)(2k+2)(2k+3)(2k+4)}{(4k+1)(4k+2)(4k+3)(4k+4)(4k+5)(4k+6)}\rightarrow\frac{1}{64},\\
\frac{\binom{4k}{2k+2}}{\binom{4k+6}{2k+2}} &= \frac{(2k-1)(2k)(2k+1)(2k+2)(2k+3)(2k+4)}{(4k+1)(4k+2)(4k+3)(4k+4)(4k+5)(4k+6)}\rightarrow\frac{1}{64},\\
\frac{\binom{4k}{2k+3}}{\binom{4k+6}{2k+2}} &= \frac{(2k-2)(2k-2)(2k)(2k+1)(2k+2)(2k+4)}{(4k+1)(4k+2)(4k+3)(4k+4)(4k+5)(4k+6)}\rightarrow\frac{1}{64}.
\end{align*}
Hence, when $k\rightarrow \infty$, we have
\begin{align*}
\frac{|\mathcal{A}_m|}{\binom{4k+6}{2k+2}}&= \frac{7\binom{4k}{2k}+\binom{7}{2}\binom{4k}{2k+1} +\binom{7}{1}\binom{4k}{2k+2} +\binom{4k}{2k+3}}{\binom{4k+6}{2k+2}}\rightarrow\frac{36}{64}=0.5625.
\end{align*}

Hence, we have the following result:

\begin{theorem} \label{T:2}
If $n=4k+7$ and  $m=2k+3$, then
\begin{align*}
\textstyle
\lim_{k\rightarrow \infty}\frac{\mathrm{Spec}(k,2)}{\binom{4k+6}{2k+2}} \leq  0.5625.
\end{align*}
\end{theorem}

We post the following conjecture:
\begin{conjecture}\label{Co:1}
Can we find $\lambda$ such that $0<\lambda<\frac{1}{2}$ and
\[
\textstyle
     \lim_{k\rightarrow \infty}\frac{\mathrm{Spec}(k,2)}{\binom{4k+6}{2k+2}}<\lambda?\]
\end{conjecture}

We post a related conjecture regarding the size of an $m$-bicommutative system for finite values of $n$ and $m$ with $m<n$:
\begin{conjecture} \label{Co:2}
Let $n=4k+7$, $m=2k+3$ and $k\geq 2$. Does  the $m$-bicommutative system   $\mathcal{A}_m$ have the minimum size among all $m$-bicommutative systems in $\mathcal{P}_n(m)$?
\end{conjecture}

\section{Constructions of a maximal commutative algebraic system  in $\mathcal{P}_n(*)$ when $n=4k+7$ with $k\geq 2$} \label{Section3}

Let $n=4k+7$, where $k\geq 2$. By our notations in Section \ref{section2} we write  $\mathcal{S}_7' = [4k+1 ,4k+7]$, so $[n]=\mathcal{S}_7' \cup {[4k]}$.  Define
\[
\mathrm{Cone}_{2k+1} := \Big\{ S\text{  }|\text{  }S\subset{[4k]} \text{ such that }|S|=2k+1\Big\}.
\]
We recall the definitions of $\mathcal{B}_8$, $\mathcal{B}_9$ and $\mathcal{B}_{10}$ defined in Section \ref{section2} also.
Define \[\mathrm{Cone}_{2k+3}:=\cup_{i=1}^{7}\mathcal{B}_i.\]We show the following intersecting properties of $\mathrm{Cone}_{2k+1}$ and $\mathrm{Cone}_{2k+3}$.

\begin{lemma} \label{L:6}
Let $S \subset [4k+7]$ such that $|S|=2k+1$ and $S\notin \mathrm{Cone}_{2k+1}$. There always exists $S'\in \mathrm{Cone}_{2k+1} \cup\mathrm{Cone}_{2k+3}$ such that $S\cap S'=\emptyset$.
\end{lemma}

\begin{proof}
Clearly  $[n]=\mathcal{S}_7' \cup {[4k]}$. For $S\subset [n]$ and $|S|=2k+1$, we write $S=S_0 \cup S_1$ where $S_0\subset\mathcal{S}_7'$ and $S_1\subset{[4k]}$. If $S\notin\mathrm{Cone}_{2k+1}$, then $|S_0|>0$.

If $|S_0|\geq 2$, then $|S_1|\leq 2k-1$ and $|{[4k]} \setminus S_1|\geq 2k+1$. We pick $2k+1$ elements in ${[4k]} \setminus S_1$ and denoted  this set by $S_2$. Clearly $S_2\in\mathrm{Cone}_{2k+1}$ and $S_2 \cap S=\emptyset$.

If $|S_0|=1$, then $|S_1|=2k$ and $|{[4k]} \setminus  S_1|=2k$. We let $S_3={[4k]} \setminus S_1$. In $\mathcal{S}_7' \setminus S_0$, there must exist an element in the $k$-Fano System, denote  it by $S_4$. We note that $S_4\cup S_3 \in \mathrm{Cone}_{2k+3}$ and $(S_4 \cup S_3)\cap S =\emptyset$.
\end{proof}

\begin{lemma} \label{L:7}
$\vee^1 \mathrm{Cone}_{2k+1} = \mathrm{Cone}_{2k+1}\cup \mathcal{B}_8 \cup \mathcal{B}_9 \cup \mathcal{B}_{10}$.
\end{lemma}

\begin{proof}
Let $S_0\subset[1 ,4k+7]$ such that $|S_0|=2$ and $S_0 \cap S=\emptyset$ for all $S\in\mathrm{Cone}_{2k+1}$. So $|S\cup S_0|=2k+3$. If $S_0 \subset {[4k]}$, then $S\cup S_0 \in\mathcal{B}_{10}$. If one element of $S_0$ is in ${[4k]}$, then $S\cup S_0\in\mathcal{B}_9$. Otherwise, $S_0\subset \mathcal{S}_7'$ and $S \cup S_0 \in\mathcal{B}_8$.
\end{proof}

As a consequence of the previous lemma we get that
%\[
\begin{equation}\label{E:3}
(\vee^1 \mathrm{Cone}_{2k+1})\cup\mathrm{Cone}_{2k+3}=\mathrm{Cone}_{2k+1} \cup (\cup_{i=1}^{10}\mathcal{B}_i).
\end{equation}
%\]

\begin{lemma} \label{L:8}
Let $S \subset[n]$ such that $|S|=2k+3$ and $S\notin (\vee^1 \mathrm{Cone}_{2k+1})\cup\mathrm{Cone}_{2k+3}$. There always exists $S'\in (\vee^1 \mathrm{Cone}_{2k+1})\cup\mathrm{Cone}_{2k+3}$ such that $S\cap S'=\emptyset$.
\end{lemma}

\begin{proof}
It is clear by Theorem \ref{T:1}.
\end{proof}

Next, we look at elements $S$ of size $2k+5$ in $\vee^1 (\cup_{i=1}^{10}\mathcal{B}_i)$. Define
\[
\mathrm{Cone}_{2k+5}:=\{S\text{  }|\text{  }S\in \vee^1(\cup_{i=1}^{10}\mathcal{B}_i)\text{  and }|S|=2k+5   \}.
\]

\begin{lemma}  \label{L:9}
Let $S\subset [n]$ such that $|S|=2k+5$ and $S\notin \mathrm{Cone}_{2k+5}$. There always exists $S'\in \vee^2 (\mathrm{Cone}_{2k+1})\cup \vee^1 (\mathrm{Cone}_{2k+3})$ such that $S\cap S'=\emptyset$.
\end{lemma}

\begin{proof}
We note that
\[
\begin{split}
\vee^2 (\mathrm{Cone}_{2k+1})\cup \vee^1 (\mathrm{Cone}_{2k+3}) &=\vee^1((\vee^1 \mathrm{Cone}_{2k+1})\cup\mathrm{Cone}_{2k+3})\\
&= \vee^1(\mathrm{Cone}_{2k+1} \cup (\cup_{i=1}^{10}\mathcal{B}_i)).
\end{split}
\]
Let $S'\in\mathrm{Cone}_{2k+5}$. We write $S' =S_0\cup S_1$ where $S_0\subset\mathcal{S}_7'$ and $S_1\subset{[4k]}$. It is clear that $S'\in \vee^1(\cup_{i=8}^{10}\mathcal{B}_i)$ if $0\leq |S_0|\leq 4$ and  $S'\notin \vee^1(\cup_{i=1}^{10}\mathcal{B}_i)$ if $|S_0|\in [6,7]$. If $|S_0|=5$, then $|S_1|=2k$ and $S_0$ always contains at least one element in the $k$-Fano System, call it $F$. Hence, \[S'=S_0\cup S_1\in(\vee^1 F )\cup S_1\subset \vee^1 (F \cup S_1)\subset \vee^1(\cup_{i=1}^{10} \mathcal{B}_i).
\]
Consequently, $S'\in\mathrm{Cone}_{2k+5}$ if and only if $0\leq |S_0|\leq 5$.

So, if $S\notin\mathrm{Cone}_{2k+5}$ then we set $S=S_0\cup S_1$,  where $S_0\subset\mathcal{S}_7'$ and $S_1\subset{[4k]}$. It is clear that  $|S_0|\in [6,7]$and it  implies (keeping   the order) that
\[
|S_1|\in  [2k-1, 2k-2] \quad \text{ and}\quad   |{[4k]} \setminus S_1|\in [2k+1,  2k+2]
\]
 respectively. We pick $2k+1$ elements in ${[4k]} \setminus S_1$ and denote this set by $S_3$. We note that $S_3\in\mathrm{Cone}_{2k+1}$ and $S\cap S_3=\emptyset$.
\end{proof}

Next, we look at the structure of $\mathrm{Cone}(\mathrm{Cone}_{2k+1}\cup\mathrm{Cone}_{2k+3})$.

\begin{lemma}   \label{L:10}
If $S\subset [n]$, $|S|$ is odd and $|S|\geq 2k+7$, then
\[
S\in \mathrm{Cone}(\mathrm{Cone}_{2k+1}\cup \mathrm{Cone}_{2k+3}).
\]
\end{lemma}

\begin{proof}
We only need to look at the case $|S|=2k+7$. Write $S=S_0\cup S_1$, where $S_0\subset \mathcal{S}_7'$ and  $S_1\subset {[4k]}$. Then there always exists $S'=S_0' \cup S_1'$, in which $S_0'\subset \mathcal{S}_7'$ and  $S_1'\subset {[4k]}$ such that $|S'|=2k+5$, $|S_0'|\leq 5$ and $S\in \vee^1\ S'$. However
\[
S' \in \mathrm{Cone}_{2k+5}\subset \vee^2(\mathrm{Cone}_{2k+1} \cup \mathrm{Cone}_{2k+3})
\]
based on  the proof of   previous lemma.
\end{proof}

\begin{theorem}
Let $\mathrm{Cone}=\mathrm{Cone}(\mathrm{Cone}_{2k+1} \cup \mathrm{Cone}_{2k+3})$. Then $\mathrm{Cone}$ is a maximal commutative algebraic system in $\mathcal{P}_{4k+7}(*)$ for $k\geq 2$.
\end{theorem}

\begin{proof}
Clearly  $[n] = \mathcal{S}_7' \cup {[4k]}$. If $S,S'\in\mathrm{Cone}$ such that $|S|+|S'|> 4k+7$, then $S\cap S'\not=\emptyset$ by Pigeonhole Principle. So we need to consider the intersecting properties for $S$ and $S'$ in the case $|S|\leq 2k+3$ and $|S'|\leq 2k+3$.

If $S, S'\in\mathrm{Cone}$ such that $|S|=|S'|=2k+1$, then $S,S'\in \mathrm{Cone}_{2k+1}$ and hence $S\cap S'\not=\emptyset$ in $[4k]$ by Pigeonhole Principle.

Note that any element of size $2k+3$ in $\mathrm{Cone}$ belongs  to $\cup_{i=1}^{10} \mathcal{B}_i$ by \eqref{E:3}. If $S,S'\in\mathrm{Cone}$ such that $|S|=|S'|=2k+3$, then $S\cap S'\not=\emptyset$  by Theorem \ref{T:1}.

If $S,S'\in\mathrm{Cone}$ such that $|S|=2k+1$ and $|S'|=2k+3$. We only need to consider the case when $S\in\mathrm{Cone}_{2k+1}$ and $S'\in\mathrm{Cone}_{2k+3}$. However  $S\cap S'\not=\emptyset$ in ${[4k]}$ by Pigeonhole Principle. Hence $\mathrm{Cone}$ is commutative.

Evidently,  $\mathrm{Cone}$ is an algebraic system due to Lemma \ref{L:2}.

For the maximality of $\mathrm{Cone}$, we only need to consider $S\subset[n]$ such that
\[
|S|\quad  \text{ is odd\quad  and} \quad |S|\leq 2k-1
\]
 by Lemma \ref{L:6}, Lemma \ref{L:8}, Lemma \ref{L:9} and Lemma \ref{L:10}.

In this case, there always exists $S'=S_0' \cup S_1'$, such that $S_0'\subset\mathcal{S}_7'$ and  $S_1'\subset{[4k]}$ for which  $|S'|=2k+1$,  $S_0' \neq \emptyset$ and  $S\subset S'$. (In other words, we can always construct such $S'$ by adding more elements to $S$ with conditions  $S_0'\neq\emptyset$ and $|S'|=2k+1$.) Note that $S'\notin \mathrm{Cone}_{2k+1}$. However  there exists
\[
S''\in\mathrm{Cone}_{2k+1}\cup\mathrm{Cone}_{2k+3} \subset\mathrm{Cone}
\]
 such that $S''\cap S'=\emptyset$ by Lemma \ref{L:6}. Consequently,    $S\cap S'' \subset S'\cap S'' =\emptyset$.
\end{proof}

\begin{remark}
In the decomposition $[n]=\mathcal{S}_7' \cup [4n]$, there are $\binom{n}{7}$  choices  of  7 elements for $\mathcal{S}_7'$. For each such  choice of $\mathcal{S}_7'$, we  can construct a maximal commutative algebraic system $\mathrm{Cone}(\mathrm{Cone}_{2k+1}\cup \mathrm{Cone}_{2k+3})$.
\end{remark}

We write $\mathrm{Cone} =\mathrm{Cone}(\mathrm{Cone}_{2k+1}\cup \mathrm{Cone}_{2k+3})=\cup_i \mathcal{U}_i$ where $\mathcal{U}_i\subset P_n(i)$. Clearly,  $\mathcal{U}_i\neq \emptyset$ if and only if $i=2k+1,2k+3,\ldots ,4k+7$. More precisely,
\begin{align*}
\mathcal{U}_{2k+1}&=\mathrm{Cone}_{2k+1},\quad \mathcal{U}_{2k+3}=\cup_{i=1}^{10} \mathcal{B}_i,\\
\mathcal{U}_{2k+5}&=\{S\text{  }|\text{  }S=S' \cup S''\text{ such that }0\leq |S'| \leq 5\text{ and }S'\subset\mathcal{S}_7', S''\subset{[4k]} \},\\
\mathcal{U}_j &= \mathcal{P}_n(j), \text{ for }j\geq 2k+7\text{ and }j\text{ is odd},
\end{align*}
where the last equation  is based on Lemma \ref{L:10}.

It follows  that
\[
\begin{split}
|\mathcal{U}_{2k+1}| &= \textstyle \binom{4k}{2k+1},\\
\textstyle
|\mathcal{U}_{2k+3}| &= 7\textstyle \binom{4k}{2k}+\binom{7}{2} \binom{4k}{2k+1} +\binom{7}{1} \binom{4k}{2k+2} +\binom{4k}{2k+3}, \\
|\mathcal{U}_{2k+5}| &=\textstyle  \sum_{i=0}^5 \binom{7}{i} \binom{4k}{2k+5-i},\\
|\mathcal{U}_{j}| &= \textstyle \binom{4k+7}{j}, \text{   }j\geq 2k+7\text{ and }j\text{ is odd}.
\end{split}
\]
As a result, we have
\begin{align*}
|\mathrm{Cone}| &= |\mathcal{U}_{2k+1}|+|\mathcal{U}_{2k+3}|+|\mathcal{U}_{2k+5}| +\sum_{j\geq 2k+7, j\text{ is odd}}|\mathcal{U}_{2k+7}|.
\end{align*}
We checked  the values of $|\mathrm{Cone}|=|\mathrm{Cone}(\mathrm{Cone}_{2k+1}\cup \mathrm{Cone}_{2k+3})|$ up to $k=249$ numerically.  The calculation shows that
\[
0.9\cdot 2^{n-2}<|\mathrm{Cone}|< 2^{n-2},
\]
but $|\mathrm{Cone}|$ gets sufficiently close to $2^{n-2}$ as $k$ gets bigger. More precisely, let $[s_k]_{k\geq 2}$ be the following sequence of real numbers:
\[
\textstyle
s_k=\frac{|\mathrm{Cone}|}{2^{n-2}}
\]
where $n=4k+7$ and $k\geq 2$. It is an increasing sequence with $s_2=0.97437$, $s_{249}=0.99763$,  $s_k<1$ for $k\in [2,249]$. So, we state the following conjecture:
\begin{conjecture}\label{Co:3}
Let $n=4k+7$  with $k\geq 2$ and $\mathrm{Cone}=\mathrm{Cone}(\mathrm{Cone}_{2k+1}\cup\mathrm{Cone}_{2k+3})$. Is  the following equation  true:
\[
\textstyle
\lim_{n\rightarrow \infty}\frac{|\mathrm{Cone}|}{2^{n-2}}=1? \]
\end{conjecture}
On the other hand, we construct another  maximal commutative algebraic system, $\mathcal{D}$, based on $\mathcal{A}_m=\cup_{i=1}^{10} \mathcal{B}_i$ in Section \ref{section2}.

Let $\mathcal{D} =\cup_i \mathcal{D}_{i}$ where $\mathcal{D}_i\subset \mathcal{P}_n(i)$ and
\[
\mathcal{D}_{j}=
\begin{cases}
\cup_{i=1}^{10} \mathcal{B}_i, & \text {if } j=2k+3;\\
\mathcal{P}_n(j), &  \text{if }j\geq 2k+5\text{ and }j\text{ is odd};\\
\emptyset, &  \text{otherwise}.\\
\end{cases}
\]

\begin{theorem}\label{T:4} The set
$\mathcal{D}$ is a maximal commutative algebraic system in $\mathcal{P}_{4k+7}(*)$ for $k\geq 2$.
\end{theorem}

\begin{proof}
It is clear that $\mathcal{D}$ is an algebraic system.  Moreover      $S_1\cap S_2\neq \emptyset$  for   each pair of $S_1, S_2\in \mathcal{D}$ for which   $|S_1|+|S_2|\geq n+1=4k+8$ by Pigeonhole Principle. So we only need to check
$S_1\cap S_2\neq \emptyset$ only  for $S_1, S_2\in\mathcal{D}_{2k+3}$. However,   $S_1\cap S_2\neq \emptyset$ by Theorem \ref{T:1}, so $\mathcal{D}$ is commutative.

For $S\notin \mathcal{D}$ and $|S|=2k+3$, there exists $S'\in\mathcal{D}_{2k+3}\subset\mathcal{D}$ such that $S' \cap S=\emptyset$ by Theorem \ref{T:1}. For $S\notin\mathcal{D}$, $|S|$ is odd and $|S|<2k+3$, we note that
\[
|[n]\setminus S|>n-(2k+3)=2k+4.
\]
We randomly choose one element $s\in [n]\setminus S$ and set $S''=[n] \setminus  (S\cup \{s\})$. Note that $|S''|$ is odd and $|S''|>2k+3$. So, $S''\in\mathcal{D}_j$ for some odd $j\geq 2k+5$. It follows  that $S'' \cap S=\emptyset$ and hence $\mathcal{D}$ is maximal.
\end{proof}

We have constructed two  maximal commutative algebraic systems, namely $\mathrm{Cone}$ and $\mathcal{D}$, based on $\mathcal{A}_m$. The following property holds.

\begin{lemma} \label{L:11}
Let $n=4k+7$, $m=2k+3$ where $k\geq 2$.   Then $|\mathrm{Cone}|<|\mathcal{D}|$, where $\mathrm{Cone}$ and $\mathcal{D}$ are maximal commutative algebraic systems    based on $\mathcal{A}_m$ in $\mathcal{P}_{n}(*)$ constructed above.
\end{lemma}

\begin{proof}
Note that $\mathrm{Cone}$ and $\mathcal{D}$ have  the same collection of elements $S\in \mathcal{P}_n(i)$,  if either $i\leq 2k-1$ or $i=2k+3$ or $i\geq 2k+7$. Thus
\[
\begin{split}
|\mathcal{D}|-|\mathrm{Cone}| =&
\textstyle
 \binom{7}{7}\binom{4k}{2k-2} + \binom{7}{6} \binom{4k}{2k-1} -\binom{4k}{2k+1}\\
=&
\textstyle
\binom{4k}{2k-2}+ 6\binom{4k}{2k-1}>0.\\
\end{split}
\]
\end{proof}

\begin{remark}
If Conjecture \ref{Co:2} is true, then, by Lemma \ref{L:11} and Corollary \ref{C:1}, our construction of $\mathrm{Cone}$ would be a counterexample to a conjecture made by Domokos and Zubor (Conjecture 7.4 in \cite{Domokos_Zubor}) for $n=4k+7$ with  $k\geq  2$.
\end{remark}

We checked  the values of $|\mathcal{D}|$ up to $k=249$ numerically.  The calculation shows that
\[
2^{n-2}<|\mathcal{D}|< 1.019\cdot 2^{n-2},
\]
but $|\mathcal{D}|$ gets sufficiently close to $2^{n-2}$ as $k$ gets bigger.
More precisely, let $[d_k]_{k\geq 2}$ be the following sequence of real numbers:
\[
\textstyle
d_k=\frac{|\mathcal{D}|}{2^{n-2}}\qquad (n=4k+7, k\geq 2)
\]
It is a decreasing sequence with $d_2=1.0188$, $d_{249}=1.0031$ and $1<d_k$ for $k\in [2,249]$. It leads us to the following conjecture:

\begin{conjecture}\label{Co:4}
Let $n=4k+7$. Is the following equation true:
\[
\textstyle
\lim_{n\rightarrow \infty}\frac{|\mathcal{D}|}{2^{n-2}}=1?
\]
\end{conjecture}

\section{Constructions of a  maximal commutative algebraic system in $\mathcal{P}_n(*)$ when $n=4k+9$ with  $k\geq 2$} \label{Section4}

Let $n=4k+9$ and $k\geq 2$. We write  $\mathcal{T}_7' = [4k+3 ,4k+9]$. So, $[n]=\mathcal{T}_7' \cup {[4k+2]}$. We set
\[
\mathcal{A}^{4k+2}_3 :=\{A^{4k+2}_1, A^{4k+2}_2,\ldots ,A^{4k+2}_7 \},
\]
where
\[
\begin{split}
A^{4k+2}_1&=\{4k+3, 4k+4, 4k+7\},\quad  A^{4k+2}_2=\{4k+3, 4k+5, 4k+8\},\\
A^{4k+2}_3&=\{4k+3, 4k+6, 4k+9\},\quad  A^{4k+2}_4=\{4k+4, 4k+5, 4k+9\},\\
A^{4k+2}_5&=\{4k+5, 4k+6, 4k+7\},\quad  A^{4k+2}_6=\{4k+7, 4k+8, 4k+9\}\\
\end{split}
\]
and $A^{4k+2}_7=\{4k+4, 4k+6, 4k+8\}$. Based on the notation introduced in Section \ref{section2}, $\mathcal{A}^{4k+2}_3$ is a {\it ($4k+2$)-Fano System}. By analogy with  Lemma \ref{L:3}, we have.

\begin{lemma} \label{L:12}
The ${4k+2}$-Fano System $\mathcal{A}^{4k+2}_3$ is a $3$-bicommutative system in $\mathcal{P}_{\mathcal{T}_7'}(3)$.
\end{lemma}

Let
\begin{equation}\label{MFdJL}
\mathcal{C}_{2k+3}:=\cup^{8}_{i=1} \mathcal{F}_i
\end{equation}
 where $\mathcal{F}_i$ is a collection of subsets of size $2k+3$ in $[n]$ as follows: \begin{align*}
\mathcal{F}_i &= \{ S\mbox{ }|\mbox{ }S=A^{4k+2}_i \cup S'\mbox{ such that }S'\subset [4k+2], |S'|=2k\}, \mbox{       }1\leq i\leq 7,\\
\mathcal{F}_8&= \{ S\mbox{ }|\mbox{ }S\subset [4k+2]\mbox{ and }|S|=2k+3\}.
\end{align*}

Let $\mathcal{P}'_7(4)$ be the collection of all subsets of size $4$ in $\mathcal{T}'_7$. $|\mathcal{P}'_7(4)| = 35$. Let
\[
\begin{split}
E_1&=\{ 4k+3, 4k+4, 4k+8, 4k+9\},\quad E_2=\{4k+4, 4k+5, 4k+7, 4k+8\},\\
E_3&=\{4k+5, 4k+6, 4k+8, 4k+9\},\quad E_4=\{4k+3, 4k+5, 4k+7, 4k+9\},\\
E_5&=\{4k+3, 4k+4, 4k+5, 4k+6\},\quad E_6=\{4k+3, 4k+6, 4k+7, 4k+8\},\\
E_7&=\{4k+4, 4k+6, 4k+7, 4k+9\}.
\end{split}
\]We set $\mathcal{Q}_7(4)\subset\mathcal{P}'_7(4)$  as follows:
\[
\mathcal{Q}_7(4):=\mathcal{P}'_7(4) \setminus \{E_1, E_2, E_3, E_4, E_5, E_6, E_7 \}.
\]
We note that $|\mathcal{Q}_7(4)|=28$ and  the elements in $\mathcal{Q}_7(4)$ are labeled randomly as $Q_1, Q_2, \ldots, Q_{28}$.

The next three lemmas involve finite combinatorics and can be checked easily.

\begin{lemma} \label{L:13}
For any $S_1\in\mathcal{A}^{4k+2}_3$ and $S_2\in\mathcal{Q}_7(4)$, we have $S_1\cap S_2\neq \emptyset$.
\end{lemma}

\begin{lemma} \label{L:14}
For any $S_1\in\mathcal{A}^{4k+2}_3$ and $s\in\mathcal{T}'_7$ such that $S_1\cap\{s\}=\emptyset$, the collection of all $S_1\cup\{s\}$ is the equal to $\mathcal{Q}_7(4)$. That is,
\[
\{ S_1\cup\{s\} \quad |\quad  S_1\in\mathcal{A}_3^{4k+2}, \quad s\in\mathcal{T}'_7, \quad S_1\cap\{s\}=\emptyset\}=\mathcal{Q}_7(4).
\]
\end{lemma}

\begin{lemma} \label{L:14.5}
For any $E_i$ where $i\in [7]$, $\mathcal{T}_7'\setminus E_i = A_j^{4k+2}$ where $A_j^{4k+2}\in \mathcal{A}_3^{4k+2}$ for some $j\in [7]$.
\end{lemma}

Let
\begin{equation}\label{DeXppF}
\mathcal{C}_{2k+5}:=\cup^{39}_{i=1} \mathcal{D}_i
\end{equation}
where $\mathcal{D}_i$ is a collection of subsets of size $2k+5$ in $[n]$ as follows:
\begin{align*}
\mathcal{D}_1 &= \{S\mbox{ }|\mbox{ }S\subset [4k+2]\mbox{ and }|S|=2k+5\}, \\
\mathcal{D}_2 &= \{S\mbox{ }|\mbox{ }S=S'\cup S''\mbox{ such that } S'\subset\mathcal{T}'_7, S''\subset [4k+2], |S'|=1, |S''|=2k+4\},\\
\mathcal{D}_3 &= \{S\mbox{ }|\mbox{ }S=S'\cup S''\mbox{ such that } S'\subset\mathcal{T}'_7, S''\subset [4k+2], |S'|=2, |S''|=2k+3\},\\
\mathcal{D}_i &=\{S\mbox{ }|\mbox{ }S=A^{4k+2}_{i-3} \cup S''\mbox{ such that }S''\subset [4k+2], |S''|=2k+2\}, \mbox{        }4\leq i\leq 10,\\
\mathcal{D}_i &= \{S\mbox{ }|\mbox{ }S=Q_{i-10} \cup S''\mbox{ such that }S''\subset [4k+2], |S''|=2k+1\}, \mbox{        }11\leq i \leq 38,\\
\mathcal{D}_{39} &=\{S\mbox{ }|\mbox{ }S=S'\cup S''\mbox{ such that } S'\subset\mathcal{T}'_7, S''\subset [4k+2], |S'|=5, |S''|=2k\}.
\end{align*}
If $i$ is odd and $n\geq i\geq 2k+7$, let
\begin{equation}\label{TDrSM}
\mathcal{C}_{i}: =\mathcal{P}_n(i)=\mathcal{P}_{4k+9}(i).
\end{equation}

We look at the union of collections of subsets $\mathcal{C}_i$ from \eqref{MFdJL}, \eqref{DeXppF} and \eqref{TDrSM}
\[
\mathcal{C}:=\cup_i \mathcal{C}_i,
\]
in which  $i$ is odd and $i\geq 2k+3$.

\begin{lemma} \label{L:15}
The set $\mathcal{C}$ is commutative in $\mathcal{P}_{4k+9}(*)$ for $k\geq 2$.
\end{lemma}

\begin{proof}
By Pigeonhole principle, $S_1\cap S_2\neq \emptyset$ for any $S_1\in\mathcal{C}_i$, $S_2\in\mathcal{C}_j$ with  $i,j\geq 2k+5$.

If $S\in\mathcal{C}_{2k+3}=\cup^8_{i=1} \mathcal{F}_i$, then $S\cap S_3\neq \emptyset$ for $S_3\in\mathcal{C}_{i}$, $i\geq 2k+7$, by Pigeonhole principle.

If $S\in\mathcal{F}_8\subset \mathcal{C}_{2k+3}$, then $S\subset [4k+2]$ and $|S|=2k+3$. We take any $S_4\in\mathcal{C}_{2k+5}$, then $S_4=S'_4\cup S''_4$ where $S'_4\subset\mathcal{T}'_7$, $S''_4\subset[4k+2]$ and $|S''_4|\geq 2k$. We note that $|S|+|S''_4|\geq 4k+3$ and hence $S\cap S''_4 \neq \emptyset$ in $[4k+2]$ by Pigeonhole principle. So, $S\cap S_4\neq \emptyset$.

For each  $S_5, S_6\in \mathcal{F}_8$, $S_5\cap S_6\neq \emptyset$ in $[4k+2]$ as $S_5, S_6\subset [4k+2]$ and $|S_5|+|S_6|=4k+6>4k+2$, by Pigeonhole principle.

If $S\in \mathcal{F}_{i}$, $1\leq i\leq 7$, then $S=A^{4k+2}_i\cup S'$, $1\leq i\leq 7$ and $S'\subset [4k+2]$, $|S'|=2k$. We take any $S_7\in \cup^3_{i=1} \mathcal{D}_i\subset \mathcal{C}_{2k+5}$. $S_7=S'_7\cup S''_7$ where $S'_7\subset\mathcal{T}'_7$, $S''_7\subset [4k+2]$ and $|S''_7|\geq 2k+3$. We note that $|S'|+|S''_7|\geq 4k+3$ and hence $S'\cap S''_7\neq \emptyset$ in $[4k+2]$ by Pigeonhole principle. So, $S\cap S_7\neq \emptyset$.

If $S\in\mathcal{F}_i$, $1\leq i \leq 7$ and $S_8\in\cup^{10}_{i=4} \mathcal{D}_i\subset\mathcal{C}_{2k+5}$, then $S_8=A^{4k+2}_{i-3}\cup S''_8$ for some $4\leq i\leq 10$ and $S''_8\subset [4k+2]$. $S$ and $S_8$ have non-trivial intersection in $\mathcal{T}'_7$ by Lemma \ref{L:12}.

If $S\in\mathcal{F}_i$, $1\leq i\leq 7$ and $S_9\in \cup^{38}_{i=11} \mathcal{D}_i\subset\mathcal{C}_{2k+5}$, then $S_9=Q_{i-10}\cup S''_9$ for some $11\leq i\leq 38$ and $S''_9\subset [4k+2]$. $S$ and $S_9$ have non-trivial intersection in $\mathcal{T}'_7$ by Lemma \ref{L:13}.

If $S\in\mathcal{F}_i$, $1\leq i\leq 7$ and $S_{10}\in \mathcal{D}_{39}$, then $S$ and $S_{10}$ have non-trivial intersection in $\mathcal{T}'_7$ by Pigeonhole principle.

For $S_{11}, S_{12}\in\mathcal{F}_i$, $1\leq i\leq 7$, then $S_{11}\cap S_{12}\neq\emptyset$ in $\mathcal{S}'_7$ by Lemma \ref{L:12}.

For $S_{13}\in\mathcal{F}_i$, $1\leq i\leq 7$ and $S_{14}\in\mathcal{F}_8$, then they have non-trivial intersection in $[4k+2]$ by Pigeonhole principle.
\end{proof}

\begin{lemma} \label{L:16}
The set $\mathcal{C}$ is algebraic in $\mathcal{P}_{4k+9}(*)$ for $k\geq 2$.
\end{lemma}

\begin{proof}
If $S\in\mathcal{C}_i$ where $i\geq 2k+5$ and $S'\subset [n]$ such that $|S'|$ is even and $S\cap S'=\emptyset$, then it is obvious that $S\cup S'\in\mathcal{P}_n(j) =\mathcal{C}_j$ for some $j\geq 2k+7$.

If $S\in\mathcal{F}_8\subset\mathcal{C}_{2k+3}$ and $S'\subset [n]$ such that $|S'|=2$ and $S\cap S'=\emptyset$, then $S\cup S'\in \cup^3_{i=1} \mathcal{D}_i$. If $S''\subset [n]$ such that $|S''|$ is even, $|S''|\geq 4$ and $S\cap S''=\emptyset$, then it is obvious that $S\cup S''\in\mathcal{P}_n(j)   =\mathcal{C}_j$ for some $j\geq 2k+7$.

If $S\in\mathcal{F}_i$ for $1\leq i\leq 7$ and $S'\subset [n]$ such that $|S'|=2$ and $S\cap S'=\emptyset$, then there are a few cases to consider. First, if $S' \subset [4k+2]$, then $S\cup S'\in\cup_{j=4}^{10} \mathcal{D}_{j}$. Second, if $S'=\{s_1,s_2\}$ where $s_1\in\mathcal{T}'_7$ and $s_2\in [4k+2]$, then $S\cup S' \in\cup_{j=11}^{38} \mathcal{D}_j$ by Lemma \ref{L:14}. Third, if $S'\subset \mathcal{S}'_7$, then $S\cup S'\in\mathcal{D}_{39}$.

If $S\in\mathcal{F}_i$ for $1\leq i\leq 7$ and $S'\subset [n]$ such that $|S'|$ is even, $|S'|\geq 4$ and $S\cap S'=\emptyset$, then it is obvious that $S\cup S'\in \mathcal{P}_n(j)=\mathcal{C}_j$ for some $j\geq 2k+7$.
\end{proof}

\begin{lemma} \label{L:17}
The set $\mathcal{C}$ is maximal in $\mathcal{P}_{4k+9}(*)$ for $k\geq 2$.
\end{lemma}

\begin{proof}
For any $S\subset [n]$ such that $S\notin \mathcal{C}$ and $|S|$ is odd, we need to show that there exists $S'\in\mathcal{C}$ such that $S\cap S'=\emptyset$. We only need to consider the cases when $|S|\leq 2k+5$.

If $|S|\leq 2k+1$, then $S\notin\mathcal{C}$ and
\[
|[n] \setminus S|\geq (4k+9)-(2k+1)=2k+8.
\]
 We randomly pick $2k+7$ elements in $[n] \setminus S$ and denote such collection of elements by $S_1$. Clearly, $S_1\in\mathcal{C}_{2k+7}\subset\mathcal{C}$ and $S\cap S_1=\emptyset$.

If $|S|=2k+5$ and $S\notin\mathcal{C}$, then we write $S=S_1\cup S_2$ where $S_1\subset\mathcal{T}'_7$ and $S_2\subset [4k+2]$. There are three cases to consider. First, if $|S_2|=2k-1$ or $2k-2$, then
\[
|[4k+2]\setminus S_2|\in\{2k+3, 2k+4\}.
\]
We randomly pick $2k+3$ elements in $[4k+2]\setminus S_2$ and denote this collection of elements by $T_1$. We note that $T_1\in\mathcal{F}_8\subset\mathcal{C}$ and $S\cap T_1=\emptyset$. Second, if $|S_2|=2k+1$ and $|S_1|=4$, then $S_1\notin \mathcal{Q}_7(4)$ and hence $S_1=E_i$ for some $i\in [7]$. By Lemma \ref{L:14.5}, there exists a set $T_2\in \mathcal{A}_3^{4k+2}$ such that $S_1\cap T_2=\emptyset$. Also, we randomly choose $2k$ elements in $[4n+2]\setminus S_2$ and denote this set by $T_3$. It is clear that $S_2 \cap T_3=\emptyset$. Let $T_4=T_2\cup T_3$. It is clear that $T_4\in \mathcal{F}_j \subset \mathcal{C}$ for some $j\in [7]$ and $S\cap T_4=\emptyset$. Third, if $|S_2|=2k+2$ and $|S_1|=3$, then $|[4k+2]\setminus S_2|=2k$ and $S_1\notin \mathcal{A}_3^{4k+2}$. By Lemma \ref{L:12}, there exists $T_5\in\mathcal{A}_3^{4k+2}$ such that $S_1 \cap T_5=\emptyset$. Let $T_6:=[4k+2]\setminus S_2$ and  $T_7:=T_5\cup T_6$. It is clear that $S\cap T_7=\emptyset$ and $T_7\in \mathcal{F}_{i'}\subset \mathcal{C}$ for some $i'\in [7]$.

If $|S|=2k+3$ and $S\notin \mathcal{C}$, it implies that $S\notin \mathcal{C}_{2k+3}$. We write $S=S_1\cup S_2$ where $S_1\subset\mathcal{T}'_7$ and $S_2\subset [4k+2]$. Then there are a few cases to consider.

If $|S_1|=7$ and $|S_2|=2k-4$, then $|[4k+2]\setminus S_2|=2k+6$. We randomly pick $2k+5$ elements in $[4k+2]\setminus S_2$ and denote this collection of elements by $S_3$. Clearly, $S_3\in\mathcal{D}_1\subset\mathcal{C}$ and $S\cap S_3=\emptyset$.

If $|S_1|=6$ and $|S_2|=2k-3$, then $|\mathcal{T}'_7 \setminus S_1|=1$ and $|[4k+2]\setminus S_2|=2k+5$. Let $s'\in (\mathcal{T}'_7 \setminus S_1)$. We randomly pick $2k+4$ elements in $[4k+2]\setminus S_2$ and denote this collection of elements by $S_4$. Let $S_5=\{s'\}\cup S_4$. Clearly, $S_5\in \mathcal{D}_2\subset\mathcal{C}$ and $S\cap S_5=\emptyset$.

If $|S_1|=5$ and $|S_2|=2k-2$, then $|\mathcal{T}'_7 \setminus S_1|=2$ and $|[4k+2] \setminus S_2|=2k+4$. Let $S_6= \mathcal{T}'_7\setminus S_1$. We randomly pick $2k+3$ elements in $[4k+2] \setminus S_2$ and denote this collection of elements by $S_7$. Let $S_8 = S_6 \cup S_7$. Clearly, $S_8\in\mathcal{D}_3\subset\mathcal{C}$ and $S\cap S_8=\emptyset$.

If $|S_1|=4$ and $|S_2|=2k-1$, then $|\mathcal{T}'_7 \setminus S_1|=3$ and $|[4k+2] \setminus S_2|=2k+3$. We randomly pick 2 elements in $\mathcal{T}'_7 \setminus S_1$ and denote this collection of elements by $S_9$. Let $S_{10} = [4k+2] \setminus S_2$. Let $S_{11}=S_9 \cup S_{10}$. Clearly, $S_{11}\in\mathcal{D}_3\subset\mathcal{C}$ and $S\cap S_{11}=\emptyset$.

If $|S_1|=3$ and $|S_2|=2k$, then $S_1\notin \mathcal{A}^{4k+2}_3$. By Lemma \ref{L:12}, there always exists an element $S_{12}\in\mathcal{A}^{4k+2}_3$ such that $S_1\cap S_{12}=\emptyset$. Let $S_{13}=[4k+2] \setminus S_2$. Clearly, $|S_{13}|=2k+2$. Let $S_{14}=S_{12}\cup S_{13}$. Then $S_{14}\in \cup^{10}_{j=4} \mathcal{D}_j\subset\mathcal{C}$ and $S\cap S_{14}=\emptyset$.

If $|S_1|=2$ and $|S_2|=2k+1$, then $|\mathcal{T}'_7 \setminus S_1|=5$ and let $S_{15}=\mathcal{T}'_7 \setminus S_1$. We note that $|[4k+2] \setminus S_2|=2k+1$. We randomly pick $2k$ elements in $[4k+2]\setminus S_2$ and denote such collection of elements by $S_{16}$. Clearly, $S_1\cap S_{15}=\emptyset$ and $S_2\cap S_{16}=\emptyset$. Let $S_{17}=S_{15}\cup S_{16}$. It is clear that $S_{17}\in\mathcal{D}_{39}\subset\mathcal{C}$ and $S\cap S_{17}=\emptyset$.

If $|S_1|=1$ and $|S_2|=2k+2$, then $|\mathcal{T}'_7 \setminus S_1|=6$. We randomly pick $5$ elements in $\mathcal{T}'_7 \setminus S_1$ and denote such collection of elements by $S_{18}$. Let $S_{19} = [4k+2] \setminus S_1$. We note that $|S_{19}|=2k$. Let $S_{20}=S_{18}\cup S_{19}$. Clearly, $S_{20}\in\mathcal{D}_{39}\subset\mathcal{C}$ and $S\cap S_{20}=\emptyset$.
\end{proof}

\begin{corollary} \label{C:3}
The set $\mathcal{C}$ is a maximal commutative algebraic system in $\mathcal{P}_{4k+9}(*)$ for $k\geq 2$.
\end{corollary}

\begin{proof}
It is obvious by Lemma \ref{L:15}, Lemma \ref{L:16} and Lemma \ref{L:17}.
\end{proof}

\begin{remark}
In the decomposition $[n]=[4k+9]=\mathcal{T}_7' \cup [4k+2]$, there are $\binom{n}{7}$  choices  of  7 elements for $\mathcal{T}_7'$. For each such choice of $\mathcal{T}_7'$, we  can construct a maximal commutative algebraic system $\mathcal{C}$.
\end{remark}

We count the number of elements in $\mathcal{C}=\cup_{i=2k+3}^{4k+9} \mathcal{C}_i$,
\[
\begin{split}
  |\mathcal{C}_{2k+3}|   & =|\cup_{i=1}^7 \mathcal{F}_i|+|\mathcal{F}_8|=7 \textstyle  \binom{4k+2}{2k} +\binom{4k+2}{2k+3},\\
    |\mathcal{C}_{2k+5}| &= |\mathcal{D}_1|+|\mathcal{D}_2|+|\mathcal{D}_3|+|\cup_{i=4}^{10}\mathcal{D}_i|+|\cup_{i=11}^{38}\mathcal{D}_i|+|\mathcal{D}_{39}|\\
        &=\begin{aligned}[t]
            &\textstyle \binom{4k+2}{2k+5}+\binom{7}{1} \binom{4k+2}{2k+4} + \binom{7}{2} \binom{4k+2}{2k+3} +7\binom{4k+2}{2k+2}\\
            &+28\textstyle \binom{4k+2}{2k+1}+\binom{7}{5}\binom{4k+2}{2k},
        \end{aligned}\\
|\mathcal{C}_i| &= \textstyle \binom{4k+9}{i} \mbox{  for } i\geq 2k+7\mbox{ and }i\mbox{ is odd}.
\end{split}
\]

For $n=4k+9$ and $k\geq 2$, $K_n$ is defined as the following quantity:
\[
\textstyle
K_n:=\frac{|\mathcal{C}|}{2^{n-2}}.
\]

The following lemma can be checked numerically:

\begin{lemma} \label{L:18}
$K_n<1$ for $n=4k+9$, $k\geq 2$ and $n<1000$.
\end{lemma}

\begin{remark}
We note that the sequence $\{K_{4i+1}\mbox{ }|\mbox{ }4\leq i\leq 249\}$ is an increasing sequence of values, where $K_{17}=0.98462$ and $K_{997}=0.99763$.
\end{remark}

The following corollary shows that a conjecture made by Domokos and Zubor (Conjecture 7.3 in \cite{Domokos_Zubor}) is false.

\begin{corollary} \label{C:4}
Let $n=4k+1$ such that $17\leq n<1000$. There exist maximal commutative subalgebras in the Grassmann algebra $G(n)$ with dimension less than $3\cdot 2^{n-2}$.
\end{corollary}

\begin{proof}
By Corollary \ref{C:3}, Lemma \ref{L:1} and Lemma \ref{L:18}, the maximal commutative subalgebra $G_0\oplus M_{\mathcal{C}}$ in $G(n)=G(4k+1)$ has dimension \[2^{n-1} +|\mathcal{C}|=2^{n-1} +|\cup_{i=2k+3}^{4k+9}\mathcal{C}_i|<2^{n-1}+2^{n-2}=3\cdot 2^{n-2}\]for $n=4k+1$ such that $17\leq n<1000$.
\end{proof}

We state the following two conjectures related to Lemma \ref{L:18} and Corollary \ref{C:4}.

\begin{conjecture} \label{co:5}
Is $K_n<1$ for all $n=4k+9$ with  $k\geq 2$?
\end{conjecture}

\begin{conjecture} \label{co:6}
Is the dimension of $G_0 \oplus M_{\mathcal{C}}$ minimal among all maximal commutative subalgebras of the Grassmann algebra $G(4k+9)$ for $k \geq 2$?
\end{conjecture}

\section{Acknowledgement}
We are grateful to Professor Alexander Grishkov for  valuable discussions on the topic.

\newpage

\bibliographystyle{abbrv}

\end{document}